\documentclass[a4paper]{amsart}

\usepackage{amsmath}
\usepackage{amssymb}
\usepackage{amsfonts}
\usepackage{mathrsfs}
\usepackage{array}
\usepackage{enumerate}
\usepackage[french, english]{babel}
\usepackage[all]{xy}
\usepackage[T1]{fontenc}
\usepackage{float}
\usepackage{dsfont}

\numberwithin{equation}{section}

\newcommand{\mss}{\mathscr{S}}\newcommand{\mst}{\mathscr{T}}

     \newcommand{\BF}{{\mathbb {F}}}

    \newcommand{\BQ}{{\mathbb {Q}}}

     \newcommand{\BZ}{{\mathbb {Z}}}

    \newcommand{\CC}{{\mathcal {C}}} 
     \newcommand{\CF}{{\mathcal {F}}}

    \newcommand{\CO}{{\mathcal {O}}}

     \newcommand{\fd}{{\mathfrak{d}}}

     \newcommand{\fA}{{\mathfrak{A}}} \newcommand{\fB}{{\mathfrak{B}}}
     
     \newcommand{\fF}{{\mathfrak{F}}}

     \newcommand{\fN}{{\mathfrak{N}}}

    \newcommand{\ord}{{\mathrm{ord}}}

    \renewcommand{\mod}{\ \mathrm{mod}\ }

    \theoremstyle{plain}
    \newtheorem{thm}{Theorem}[section] \newtheorem{cor}[thm]{Corollary}
    \newtheorem{lem}[thm]{Lemma}  \newtheorem{prop}[thm]{Proposition}

\newtheorem{exm}[thm]{Example}

\theoremstyle{remark} 
\theoremstyle{remark} 
\theoremstyle{remark} 

    \numberwithin{equation}{section}

    \numberwithin{equation}{section}

\DeclareFontFamily{U}{wncy}{}
\DeclareFontShape{U}{wncy}{m}{n}{<->wncyr10}{}
\DeclareSymbolFont{mcy}{U}{wncy}{m}{n}

\begin{document}

\title[Iwasawa invariants of $K_{2m}$ over $\BZ_2$-extensions]{Iwasawa Invariants of even $K$-groups of rings of integers in the $\BZ_2$-extension over real quadratic number fields}

\author[Li-Tong Deng]{Li-Tong Deng}
\address{\textit{Li-Tong Deng}, Qiuzhen College, Tsinghua University, 100084, Beijing, China}
\email{\it dlt23@mails.tsinghua.edu.cn}

\author[Yong-Xiong Li]{Yong-Xiong Li}
\address{\textit{Yong-Xiong Li}, Yanqi Lake Beijing Institute of Mathematical Sciences and Applications, No. 544, Hefangkou Village Huaibei Town, Huairou District Beijing 101408, Beijing, China}
\email{\it yongxiongli@gmail.com}

\begin{abstract}
Let $F$ be a real quadratic number field, and let $F_{cyc}$ denote its cyclotomic $\BZ_2$-extension. For each integer $n\geq0$, let $F_n$ be the unique intermediate field in $F_{cyc}$ such that $[F_n:F]=2^n$.
By studying the $2$-adic divisibility of Dirichlet $L$-series at negative integers, we derive an asymptotic formula that determines the order of the $2$-primary part of even $K$-groups of rings of integers of $F_n$ for sufficiently large $n$. As a corollary, we determine their $\lambda$ and $\mu$ invariants. We also establish a lower bound for $n$ beyond which this asymptotic formula holds. Our results have two main applications: (1) For $K=\BQ$, $\BQ(\sqrt{p})$ or $\BQ(\sqrt{2p})$ with $p\equiv\pm3\mod 8$, we determine the structure of the $2$-primary tame kernels $K_2\CO_{K_n}(2)$; (2) We explicitly determine the three Iwasawa invariants $\lambda,\mu,\nu$ for a family of real quadratic number fields, whose discriminants have arbitrarily many prime divisors. 
\end{abstract}



\maketitle

\selectlanguage{english}

\section{Introduction}

\subsection{Background}
For any integer $m$, let $\mu_{m}$ denote the group of $m$-th roots of unity. Let $p$ be a prime and set $\mu_{p^\infty}=\bigcup_{n\geq1}\mu_{p^n}$. For any number field \(\CF\), let $\CF_{cyc}$ be the cyclotomic $\BZ_p$-extension over $\CF$. This is defined as the subfield of $\CF(\mu_{p^\infty})$ fixed by the torsion subgroup of the Galois group ${\rm Gal}(\CF(\mu_{p^\infty})/\CF)$.
For each integer $n\geq0$, we define $\CF_n$ as the unique intermediate field in $\CF_{cyc}$ such that $[\CF_n:\CF]=p^n$. 
Iwasawa \cite{Iwasawa0} (see also \cite{Iwasawa}) proved the existence of integers \(\mu_\CC,\lambda_\CC,\nu_\CC\), which depend  solely on $\CF$ and $p$. These Iwasawa invariants describe the asymptotic behavior of the order of the $p$-primary part of the ideal class group. Specifically, if $e_n$  
is the exponent such that  $p^{e_n}$ is the order of the $p$-primary part \(Cl(\CF_n)(p)\) of the ideal class group of \(\CF_n\), then for $n\geq n_{\CC,p}$, we have 
\[e_n=\mu_\CC \cdot p^n+\lambda_\CC\cdot n+\nu_\CC.\]
Here $n_{\CC,p}$ is a sufficiently large positive integer. In general, determining the Iwasawa invariants $\lambda_\CC,\mu_\CC$, $\nu_\CC$, and  the precise value of $n_{\CC,p}$ is a challenging problem.  

\medskip

A significant breakthrough was made by Ferrero and Washington \cite{FW}, who proved that $\mu_\CC=0$ if $\CF$ is an abelian extension over $\BQ$. Prior to this, Iwasawa himself \cite{I0} has shown that \(\mu_\CC=0\) when $\CF=\BQ$ or an imaginary quadratic field, and $p=2$. Later, under the same assumptions as Iwasawa, Kida \cite{Kida} provided an explicit formula for the \(\lambda\)-invariants (see also \cite{Ferrero}) and established a lower bound for \(n_{\CC,2}\).   

\medskip

It's well known that the ideal class group of a number field $\CF$ can be identified with the reduced $K_0$ group of its ring of integers $\CO_\CF$. Regarding other $K$-groups of $\CO_\CF$, Quillen \cite{Quillen0} and Borel \cite{Borel} demonstrated that $K_{2m}\CO_{\CF}$ is finite for positive integer $m$ (see also \cite{Garland}). As is customary, we refer to $K_2\CO_\CF$ as the tame kernel of $\CF$.

Let $p^{e(n)}$ denote the order of the $p$-primary part $K_{2m}\CO_{\CF_n}(p)$. Coates \cite{Coates}, Ji and Qin \cite{JQ} have established the following significant result:

\begin{thm}
There exists three integers $\mu,\lambda,\nu$, depending solely on $\CF$, $m$ and $p$, such that for all sufficiently large $n\geq n_{\CF,p}$:
\[e(n)=\mu\cdot p^n+\lambda\cdot n+\nu.\]
Here, $n_{\CF,p}$ is a sufficiently large positive integer.
\end{thm}    

\bigskip

Similar to ideal class groups, a natural question arises regarding the Iwasawa invariants  $\mu,\lambda,\nu$ and $n_{\CF,p}$. Apart from Candiotti's work \cite{Candiotti}, there appears to be little progress on this question. This paper will study the invariants  $\mu,\lambda,\nu$ and $n_{\CF,p}$ under the condition that $p=2$ and $\CF$ is either $\BQ$ or a real quadratic number field. Henceforth, we set $p=2$ and denote $n_{\CF,p}$  by $n_\CF$.

\subsection{Main Theorems}

Let $d$ be a square-free positive integer, and let $K=\BQ(\sqrt{d})$. Denote by $K_{cyc}$ the cyclotomic $\BZ_2$-extension over $K$, and let $K_n$ be the unique intermediate field in $K_{cyc}$ such that $[K_n:K]=2^n$. Specifically, when $d=1$, $K=\BQ$, and when $d=2$, $K=\BQ(\sqrt{2})$, which is contained in $\BQ_{cyc}$, with $\BQ(\sqrt{2})=\BQ_1$. Since we are primarily concerned with the asymptotic behavior of $K_{{2m-2}}\CO_{K_n}(2)$ for sufficiently large $n$, we assume that $d$ is odd. Here we assume that $m$ is any positive even integer. Denote by ${\rm ord}_2$ the order function on $\overline{\BQ}$, normalized by ${\rm ord}_2(2)=1$.

\begin{thm}\label{main-thm1}
Let $d$ be a square-free positive odd integer, and let $K=\BQ(\sqrt{d})$ be a real quadratic number field. Assume that $m$ is even.  There exists integers \(\mu,\lambda,\nu,n_K\) such that the order of the $2$-primary part \(K_{2m-2}\mathcal{O}_{K_n}(2)\) of \(K_{2m-2}\mathcal{O}_{K_n}\) satisfies
\[ord_2(|K_{2m-2}\mathcal{O}_{K_n}(2)|)=\mu\cdot 2^n+ \lambda\cdot n+\nu\]
for all \(n\ge n_K\). Precisely, we have $\mu=2$ if $\ord_2(m)=1$, $\mu=0$ if $4$ divides $m$; $\nu=\nu'_{m,d}+2+{\rm ord}_2(m)$ where $\nu'_{m,d}$ is defined in Theorem \ref{zeta-2}, and
\[\lambda=-1+\sum_{p\mid d}2^{f_p},\]
where, for any odd prime $p$, we denote by $f_p={\rm ord}_2\left(\frac{p^*-1}{4}\right)$, with $p^*=\left(\frac{-1}{p}\right)p$. Here $\left(\frac{\cdot}{p}\right)$ is the Legendre symbol.
\end{thm}

It is important to note that the even $K$-group case differs significantly from the ideal class group case due to the positivity of the $\mu$-invariant. 

\medskip

Let $m=2$. We now consider several  results regarding the structure of tame kernels in $K_{cyc}/K$.
The positivity of the $\mu$-invariant in the case of tame kernels arises from the complete splitting of the archimedean primes. If we replace the tame kernel $K_2\CO_{K_n}$ with the regular kernel 
$R_2\CO_{K_n}$ using the exact sequence
\[\xymatrix{1 \ar[r] & R_2\mathcal{O}_{K_n} \ar[r] & K_2\mathcal{O}_{K_n} \ar[r] & (\mathbb{Z}/2\mathbb{Z})^{[K_n:\mathbb{Q}]} \ar[r] & 1},\]
as stated in \cite[Exact sequence (1.3)]{Gras2}, we see that the  $2$-primary parts of the regular kernels have a vanishing $\mu$-invariant.

By a theorem of Browkin \cite[Theorem 3]{Browkin}, we can determine the structure of \(K_2\mathcal{O}_{K_n}(2)\) in the following corollary.

\begin{cor}\label{cor1-1}
For any nonnegative integer $n$, we have: 

(1) If \(K=\mathbb{Q}\) is the rational number field, then
\[K_2\mathcal{O}_{K_n}(2)\simeq (\mathbb{Z}/2\mathbb{Z})^{2^{n}}.\]

(2) If \(F=\mathbb{Q}(\sqrt{p})\) or \(\mathbb{Q}(\sqrt{2p})\) where \(p\equiv \pm 3\ (mod\ 8)\) is a prime, then
\[K_2\mathcal{O}_{F_n}(2)\simeq (\mathbb{Z}/2\mathbb{Z})^{2^{n+1}}\]
In particular, $R_2\CO_{K_n}(2)$ (resp. $R_2\CO_{F_n}(2)$) is trivial in these cases.
\end{cor}

Our second main result is the following example, in which we explicitly compute the $\lambda$, $\mu$ and $\nu$ invariants of the tame kernels for a family of real quadratic number fields whose discriminants can have arbitrarily many prime divisors. Let $d$ be a square free positive integer, and define $\psi_{d}$ the Dirichlet character associated to $\BQ(\sqrt{d})/\BQ$. Denote by $L(\psi_d,s)$ the Dirichlet $L$-series associated with 
$\psi_d$. It is well known that $L(\psi_d,-1)$ is a non-zero rational number.

\begin{exm}\label{main-thm2}
Let \(d=p_1\cdots p_r\), where the primes \(p_i\equiv 5\ (mod\ 8)\) are distinct and satisfy \((\frac{p_i}{p_j})=-1\) for \(i\neq j\), and \(r\) is odd. Let \(K=\mathbb{Q}(\sqrt{d})\). Then
\[ord_2(|K_2\mathcal{O}_{K_n}(2)|)=2\cdot 2^n+(r-1)\cdot n+({\rm ord}_2(L(\psi_{2d},-1))-1) \qquad \forall n\ge 1.\]
In particular, $\mu=2$, $\lambda=r-1$, and $\nu={\rm ord}_2(L(\psi_{2d},-1))-1$. 
\end{exm}

\bigskip

\subsection{Outline of the Proof of the Main Theorem}

Let $\BQ_{cyc}$ be the cyclotomic $\BZ_2$-extension of $\BQ$, and let $\BQ_n$ denote the unique intermediate field in $\BQ_{cyc}$ such that $[\BQ_n:\BQ]=2^n$. Denote by $\chi_n$ a primitive character of ${\rm Gal}(\BQ_n/\BQ)$. We can view $\chi_n$ as a Dirichlet character modulo $2^{n+2}$ for all $n\geq1$. The key ingredient in proving our main results is the study of the $2$-adic divisibility of the $L$-values $L(\psi_d\chi_n,1-m)$ as $n$ varies. Here $d$ is a square free positive odd integer, $\psi_d$ is the Dirichlet character associated to $\BQ(\sqrt{d})$, and $m$ is a positive even integer. The result is summarized in the following proposition.

\begin{prop}\label{2-div-psi-d}
Let $m$ be a positive even integer, and let $n$ be a positive integer. Assume that $d\geq1$ is a square-free odd integer.
Then there exists a positive integer $n_d$ such that for all $n\geq n_d$, we have
\[{\rm ord}_2(L(\chi_n\psi_d,1-m))=1+2^{1-n}\cdot\left(-1+\sum_{p\mid d}2^{f_p}\right),\]
where each $f_p$ is the constant defined in Theorem \ref{main-thm1}.
\end{prop}

We remark that when $d=1$, the equality in the proposition takes the form
\[{\rm ord}_2(L(\chi_n,1-m))=1-2^{1-n},\]
and in this case, we can show that $n_d=1$.

Our method to prove the Proposition \ref{2-div-psi-d} can be divided into the following two steps:
\begin{itemize}
  \item Step $1$: The case $d=1$. 
 In this case, we only know a priori that $L(\chi_n,1-m)$ is a nonzero algebraic number. 
 By applying the von Staudt--Clausen theorem, we can show that $L(\chi_n,1-m)$ is $2$-adically integral. This leads to a key observation:
 \begin{equation}\label{sec1-f2}
   L(\chi_n,1-m)\equiv L(\chi_n,-1)\, ({\rm mod}\,  2).
 \end{equation}
Recall that $\chi_n$ has conductor $2^{n+2}$, we then prove the congruence
\begin{equation}\label{sec1-f1}
  L(\chi_n,-1)\equiv \prod^{n}_{k=2}(1-\zeta_{2^k})\,  ({\rm mod}\,  2)\quad \textrm{for $n\geq2$,}
\end{equation}
where $\zeta_{2^k}$ is a primitive $2^k$-th root of unity satisfying $\zeta_{2^{k}}^2=\zeta_{2^{k-1}}$. Therefore, the proposition follows in the case $d=1$. Note that for any two algebraic numbers in $\BQ(\zeta_{2^n})$, the above congruences mean that their difference belongs to  $2(\BZ_2[\zeta_{2^n}]\cap \BQ(\zeta_{2^n}))$.
  \item Step $2$: The general case $d>1$.
  Let $D=2^{n+2}d$. Denote by $L^{(D)}(\chi_n,s)$ the imprimitive $L$-function asscoaited to $\chi_n$, obtained by omitting Euler factors at all primes $\ell | D$. Using a delicate special value formula modulo $4$, we derive the key congruence
      \[L^{(D)}(\chi_n,1-m)+L(\chi_n\psi_d,1-m)\equiv 0\,  ({\rm mod}\,  2(1+\zeta_4)),\]
      where $\zeta_4$ is a primitive $4$-th root of unity.
      This congruence allows us to reduce the general case to the case  $d=1$. Here the congruence also means that the sum of the two $L$-values belongs to $2(1+\zeta_4)(\BZ_2[\zeta_{2^n}]\cap \BQ(\zeta_{2^n}))$.
\end{itemize}

\bigskip

For a number field $\CF$, let $\zeta_\CF(s)$ denote the Dedekind zeta function associated to $\CF$. Proposition \ref{2-div-psi-d} implies the following corollary.

\begin{cor}\label{zeta-value-2}
Let $m$ be a positive even integer.
\begin{enumerate}
  \item [(1)]For all non-negative integer $n$, we have
  \[{\rm ord}_2(\zeta_{\BQ_n}(1-m))=2^n-n-(2+{\rm ord}_2(m)).\] 
  \item [(2)]Let $K=\BQ(\sqrt{d})$. For all integers $n\geq n_d$, we have
  \[{\rm ord}_2(\zeta_{K_n}(1-m))=2\cdot 2^n+n\cdot \left(\sum_{p\mid d}2^{f_p}-2\right)+\nu'_{m,d},\]
  where $\nu'_{m,d}$ is a constant depending only on $m$ and $d$. One can refer to Theorem \ref{zeta-2} for the precise definition.
\end{enumerate}

\end{cor}

For a totally real abelian number field $\CF$, since the Iwasawa main conjecture is known to hold (see \cite{Wiles}), the Quillen-Lichtenbaum conjecture follows. In particular, we have
\[|K_{2m-2}(\CO_\CF)|=\begin{cases}w_m(\CF)\cdot\left|\zeta_\CF(1-m)\right|\quad &\textrm{if ${\rm ord}_2(m)=1$}\\ 2^{-[\CF:\BQ]}\cdot w_m(\CF)\cdot \left|\zeta_\CF(1-m)\right|\quad &\textrm{if ${\rm ord}_2(m)\geq2$.}
\end{cases}\]
Here, $w_m(\CF)$ denotes the largest integer $\fN$ such that ${\rm Gal}(\CF(\zeta_\fN)/\CF)$ has exponent dividing $m$ (\cite{RW}).  Since the number $w_m(\CF)$ can be determined explicitly, Theorem \ref{main-thm1} follows directly from Corollary \ref{zeta-value-2}. To prove (2) in Corollary \ref{cor1-1} and Example \ref{main-thm2}, we need to show that $n_K=2$ for those fields $K=\BQ(\sqrt{d})$. This involves an induction method of $L$-values, which we borrow some ideas from \cite{DL}.

\medskip

In the final part of the introduction, we outline the structure of the paper. In Section \S\ref{sec2}, we fix the notation used throughout the paper and provide some preliminaries on character sum relations, which play a crucial role in the proof of equality \eqref{sec1-f1}. In Section \S\ref{sec3}, we present the detailed proofs of Proposition \ref{2-div-psi-d} and Corollary \ref{zeta-value-2}. In Section \S\ref{sec4}, we determine the structure of the $K$-groups, and prove Theorem \ref{main-thm1} and Example \ref{main-thm2}. In the final Section \S\ref{sec5}, we investigate the properties and behavior of the integer $n_d$ using the induction method on $L$-values. 

\medskip

\textbf{Acknowledgement}: The authors would like to thank the reviewer for the careful reading of the manuscript and for the excellent suggestions, which have improved both the mathematical content and the exposition of the paper.

\section{Notation and preliminaries}\label{sec2}

Let us fix the following notation.
\begin{itemize}
  \item For a field $\fF$, we denote by $\overline{\fF}$ an algebraic closure of $\fF$. Let $\BQ_2$ denote the field of $2$-adic numbers, and let $\BZ_2$ denote the $2$-adic integer in $\BQ_2$. Let $n$ be a nonnegative integer. For $\fA,\fB\in\BQ(\zeta_{2^n})$ and a positive rational number $\fd$, we write $\fA\equiv \fB \, ({\rm mod}\, 2^\fd)$ to mean that $\fA-\fB\in 2^{\fd}(\BZ_2[\zeta_{2^n}]\cap \BQ(\zeta_{2^n}))$.  
  \item For a number field $\CF$, let $\CO_{\CF}$ denote the ring of integers of $\CF$. Let $Cl(\CF)$ denote the ideal class group of $\CF$. For each positive integer $m$, we denote by $K_{2m-2}\CO_\CF$ the $(2m-2)$-th $K$-group of $\CO_\CF$; see \cite{Quillen0} for the definition. In the case $m=2$, the group $K_2\CO_\CF$ is known as the tame kernel of $\CF$. Let $\CF_{cyc}$ denote the cyclotomic $\BZ_2$-extension of $\CF$. For each non-negative integer $n$, we denote by $\CF_n$ the unique intermediate field in $\CF_{cyc}$ such that $[\CF_n:\CF]=2^n$.
  \item Let $\chi_n$ denote a primitive character of ${\rm Gal}(\BQ_n/\BQ)$. We can view $\chi_n$ as a Dirichlet character modulo $2^{n+2}$ for all $n\geq1$. Let $d$ be an odd square free integer, and define $\psi_d$ to be the Dirichlet character corresponding to the quadratic extension $\BQ(\sqrt{d})$. Then $\psi_d$ has conductor dividing $4d$. For a Dirichlet character $\rho$, we denote by $L(\rho,s)$ the associated Dirichlet $L$-function.  For any integer $D$, we define the imprimitive Dirichlet $L$-function by
      \[L^{(D)}(\rho,s)=\prod_{p\mid D}(1-\rho(p)p^{-s})\cdot L(\rho,s).\]
      For a number field $\CF$, we denote by $\zeta_\CF(s)$ the Dedekind zeta function of $\CF$.
  \item The Bernoulli numbers $B_n$ and the Bernoulli polynomials $B_n(X)$ are defined via the generating functions: 
\[\frac{t}{e^t-1}=\sum\limits_{n=0}^\infty B_n\frac{t^n}{n!} \qquad \frac{te^{Xt}}{e^t-1}=\sum\limits_{n=0}^\infty B_n(X)\frac{t^n}{n!}.\]
The Bernoulli polynomials satisfy the identity:
\begin{equation}\label{sec201}
  B_n(X)=\sum\limits_{k=0}^n \binom{n}{k} B_k X^{n-k}.
\end{equation}
For a Dirichlet character $\chi$ with conductor $D$, the generalized Bernoulli numbers $B_{n,\chi}$ are defined by: 
\[\sum_{a=1}^D \frac{\chi(a)te^{at}}{e^{Dt}-1}=\sum\limits_{n=0}^\infty B_{n,\chi}\frac{t^n}{n!}.\]
If $D_0$ is any multiple of $D$, from \cite[Proposition 4.1]{Washington} we have 
\begin{equation}\label{sec202}
  B_{n,\chi}=D_0^{n-1}\sum\limits_{a=1}^{D_0} \chi(a)B_n\left(\frac{a}{D_0}\right).
\end{equation}
\end{itemize}

Next, we present two special value formulas for Dirichlet  $L$-functions. These formulas establish a connection between generalized Bernoulli numbers and Dirichlet $L$-values at non-positive integers. 

\begin{lem}[\cite{DL}, Lemma 4; \cite{Washington}, Theorem 4.2]\label{bernoulli}
Let $\chi$ be a primitive Dirichlet character and $n\ge 1$. Then 
\[L(\chi,1-n)=-\frac{B_{n,\chi}}{n}\]
In particular, if $\chi$ is a non-trivial even character modulo $D_1$, then for any positive integer $k$, we have
\[L(\chi,-1)=-\frac{1}{2kD_1}\sum\limits_{a=1}^{kD_1} \chi(a)a^2.\]
\end{lem}

\begin{lem}[\cite{DL}, Proposition 5]\label{imprimitive}
Let $\chi$ be a primitive Dirichlet character modulo an even integer $d$, and let $d'$ be another odd squarefree positive integer. Put $D_1=dd'$. Then
\[L^{(D_1)}(\chi,-1)=-\frac{1}{2D_1}\sum\limits_{a=1}^{D_1} \chi(a)\psi_{d'}(a)^2a^2=L(\chi\psi_{d'}^2,-1).\]
 \end{lem}

\begin{proof}
For \(\text{Re}(s)>1\), we have
\[L^{(D_1)}(\chi,-1)=\prod\limits_{\ell \nmid dd'} (1-\chi(\ell)\ell^{-s})^{-1}=\prod\limits_{\ell \nmid d',\text{odd}} (1-\chi(\ell)\ell^{-s})^{-1}.\]
For an odd prime \(\ell\nmid d'\), we have \(\chi(\ell)=\chi(\ell)\psi_{d'}(\ell)^2\). The rest of the proof is identical with \cite[Proposition 5]{DL}.
\end{proof}

\bigskip

We now present several relations involving character sums, which will play an important role in the proof of the equation \eqref{sec1-f1}. Recall that $\chi_n$ is a primitive Dirichlet character modulo $2^{n+2}$.

\medskip

\begin{lem}\label{no-D}
Let $n$ be a positive integer, and let $d$ be a square free positive odd integer. Let $\psi$ be an even Dirichlet character modulo $4d$, and define
$\chi=\chi_n\psi$, a nontrivial even Dirichlet character modulo $D=2^{n+2}d$. Then 
\[-\frac{1}{2D}\sum\limits_{a=1}^{D} \chi(a)a^2=\frac{1}{2}\sum\limits_{a=1}^{D/2} \chi(a)a.\]
\end{lem}

\begin{proof}
Since $\chi(D-a)=\chi(a)$ for all $a$, we have
\[-\frac{1}{2D}\sum\limits_{a=1}^{D} \chi(a)a^2=-\frac{1}{2D}\sum\limits_{a=1}^{D/2} \chi(a) \bigg( a^2+(D-a)^2 \bigg).\]
Note that $\chi$ is nontrivial, expanding the right-hand side, we obtain
\[-\frac{1}{2D}\sum\limits_{a=1}^{D} \chi(a)a^2=-\frac{1}{D}\sum\limits_{a=1}^{D/2} \chi(a)a^2+\sum\limits_{a=1}^{D/2} \chi(a)a.\]
We now claim that
\[-\frac{1}{D}\sum\limits_{a=1}^{D/2} \chi(a)a^2+\frac{1}{2}\sum\limits_{a=1}^{D/2} \chi(a)a=0.\]
Once this claim is established, the lemma follows immediately.
To prove the claim, note that since $\chi$ is nontrivial, we have $\sum^{\frac{D}{2}}_{a=1}\chi(a)=0$. It is easy to verify that
\[-\frac{1}{D}\sum\limits_{a=1}^{D/2} \chi(a)a^2+\frac{1}{2}\sum\limits_{a=1}^{D/2} \chi(a)a=-\frac{1}{2D}\sum^{D/2}_{a=1}
\left(\chi(a)+\chi\left(\frac{D}{2}-a\right)\right)\cdot a^2. \] 
Now we show that
\[\chi\left(\frac{D}{2}-a\right)=-\chi(a)\]
for all integers $a$. Indeed, we may assume $a$ is odd. Otherwise, both $\chi(a)$ and $\chi\left(\frac{D}{2}-a\right)$ are zero. Recall that $\chi=\psi\chi_n$, since $\psi$ is even and has conductor dividing $4d$, $\psi\left(\frac{D}{2}-a\right)=\psi(a)$. For the character $\chi_n$,
note that
\[\chi_n\left(\frac{D}{2}-a\right)=\chi_n(a)\cdot\chi_n(1-2^{n+1}).\]
By the primitivity of $\chi_n$, we have $\chi_n(1-2^{n+1})=-1$,  so
\[\chi_n\left(\frac{D}{2}-a\right)=-\chi_n(a),\]
and hence $\chi\left(\frac{D}{2}-a\right)=-\chi(a)$. This completes the proof of the claim.  
\end{proof}

To give our second relation between character sums, we need the following

\begin{lem}\label{lem2-1}
Assume $n\ge 2$. Then for any integer $b$, we have
\begin{enumerate}
  \item[(1)]$\chi_n(2^{n+1}-b)=-\chi_n(b)$.
  \item[(2)]$\chi_n(2^n-b)=\chi_n(2^n-1)\cdot\left(\frac{-1}{b}\right)\cdot \chi_n(b)$.
\end{enumerate}
\end{lem}

\begin{proof}
We only prove (2).
We may assume that $b$ is odd. Recall that $\chi_n$  is a primitive Dirichlet character of conductor $2^{n+2}$. Then we have
\[\frac{\chi_n(2^n-b)}{\chi_n(b)\chi_n(2^n-1)}=\chi_n(1+2^n(1-b')),\]
where $b'$ is an integer such that $bb'\equiv 1\mod 2^{n+2}$. Note that $b\equiv b'\mod 4$. The lemma follows from the equality
\[\chi_n(1+2^n(1-b'))=\left(\frac{-1}{b}\right).\]
\end{proof}

Let $\zeta_4$ denote a primitive $4$-th root of unity. 

\begin{prop}\label{mod-4}
Assume that $n\ge 2$ and $\chi_n(2^n-1)=\pm \zeta_4$. Keep the assumptions as in Lemma \ref{no-D}. Then
\[-\frac{1}{2D}\sum\limits_{a=1}^{D} \chi(a)a^2\equiv (1\mp \zeta_4)\sum\limits_{a=1}^{D/8} \chi(a)\cdot a \varepsilon_a \ ({\rm mod}\ 4),\]
where $\varepsilon_a=1$ if $a\equiv 1\ ({\rm mod}\ 4)$ and $\varepsilon_a=\pm \zeta_4$ if $a\equiv 3\ ({\rm mod}\ 4)$. 
\end{prop}

\begin{proof}
From Lemma \ref{no-D} and (1) in Lemma \ref{lem2-1}, we obtain
\[-\frac{1}{2D}\sum^D_{a=1}\chi(a)a^2=\sum^{D/4}_{a=1}\chi(a)\cdot
\left(a-\frac{D}{4}\right).\]
Since $n\geq2$ so $4$ divides $D/4$. Then we get
\[-\frac{1}{2D}\sum^D_{a=1}\chi(a)a^2\equiv \sum^{D/4}_{a=1}\chi(a)a \, ({\rm mod}\, 4).\]
From (2) in Lemma \ref{lem2-1}, we have 
\begin{equation}\label{sec2-f1}
  \chi\left(\frac{D}{4}-a\right)=\pm\zeta_4\cdot\left(\frac{-1}{a}\right)\chi_n(a).
\end{equation}
Now, observe that
\[\sum^{D/4}_{a=1}\chi(a)a=\sum^{D/8}_{a=1}\chi(a)a+\sum^{D/8}_{a=1}\chi\left(\frac{D}{4}-a\right)\cdot\left(\frac{D}{4}-a\right).\]
Using \eqref{sec2-f1}, we obtain
\[\sum^{D/4}_{a=1}\chi(a)a\equiv \sum^{D/8}_{a=1}\chi(a)\cdot a\cdot \left(1\mp \zeta_4\cdot\left(\frac{-1}{a}\right)\right) \, ({\rm mod}\, 4).\]
A direct verification shows that
\[\sum^{D/8}_{a=1}\chi(a)\cdot a\cdot \left(1\mp \zeta_4\cdot\left(\frac{-1}{a}\right)\right)=(1\mp\zeta_4)\sum^{D/8}_{a=1}\chi(a)a\varepsilon_a,\]
so the proposition follows.
\end{proof}

From Lemma \ref{bernoulli} and the above proof, we obtain a corollary.

\begin{cor}\label{sec3-cor1}
For each positive integer $n$, we have
\begin{equation}\label{sec3-cor-f1}
  L(\chi_n,-1)\equiv \sum^{\frac{D}{4}}_{a=1}\chi_n(a)a\, ({\rm mod}\, 2).
\end{equation}
\end{cor}

\section{$2$-adic divisibility of Dirichlet $L$-values}\label{sec3}

We will prove Proposition \ref{2-div-psi-d} and Corollary \ref{zeta-value-2} in this section. As always, let $d$ denote a positive square-free odd integer, and put $D=2^{n+2}d$, where $n$ is a non-negative integer. According to the strategy outlined in the introduction, we first consider the case $d=1$, and then prove the general case by reducing it to this special case.

\subsection{$2$-adic divisibility of Dirichlet $L$-values for $d=1$ and $m=2$} 
In this case, note that $D=2^{n+2}$. We will prove Equalities \eqref{sec1-f1}, \eqref{sec1-f2} in this subsection by using the relations between the character sums established in the last section. Recall that we have fixed a compatible system of $2^n$-th roots of unity $\zeta_{2^n}$, satisfying $\zeta_{2^{n+1}}^2=\zeta_{2^n}$ for all $n\geq1$.

\begin{lem}\label{sec3-1}
Let $D=2^{n+2}$ and let $k$ be an integer such that $2\leq k\leq n$. For any integer $b$, we have
\[\chi_n\left(\frac{1}{2^k}D-b\right)\equiv \zeta_{2^k}\cdot\chi_n(b)\ \left({\rm mod}\ (1-\zeta_{2^{k-1}})\BZ[\zeta_{2^{n}}]\right).\]
\end{lem}

\begin{proof}
Note that the image of $\chi_n$ is equal to the group of $2^n$-th roots of unity.
For an integer $\alpha$ such that $n+2\geq \alpha\geq2$, and for an odd integer $c$, a direct calculation shows that $\chi_n(1+2^\alpha c)$ is a primitive $2^{n+2-\alpha}$-th root of unity. Let $b'$ be an integer such that $bb'\equiv 1\, ({\rm mod}\,  2^{n+2})$. Then we have
\[\frac{\chi_n(b-2^{n+2-k})}{\chi_n(b)}=\chi_n(1-2^{n+2-k}b'),\]
which is a primitive $2^k$-th root of unity. Working modulo $(1-\zeta_{2^{k-1}})\BZ[\zeta_{2^n}]$, this primitive $2^k$-th root of unity is uniquely determined. Therefore, the lemma follows.

\end{proof}

\begin{prop}\label{d-1}
Assume that $n$ is a positive integer, then $L(\chi_1,-1)=-1$ and for $n\geq2$, we have
\[L(\chi_n,-1)\equiv \prod^{n}_{k=2}(1-\zeta_{2^k})\,  ({\rm mod}\, 2\BZ[\zeta_{2^n}]).\]
In particular, 
\[{\rm ord}_2(L(\chi_n,-1))=1-2^{1-n}.\]
\end{prop}

\begin{proof}
For $n=1$, noting that
\[\zeta(-1)=-\frac{1}{12},\quad \zeta_{\BQ(\sqrt{2})}(-1)=\frac{1}{12},\]
we have $L(\chi_1,-1)=-1$. Assume that $n\geq2$. All congruence below are taken in $\BZ[\zeta_{2^n}]$. For each $2\leq k\leq n$, using the identity 
\[\sum^{\frac{D}{2^k}}_{a=1}\chi_n(a)a=\sum^{\frac{D}{2^{k+1}}}_{a=1}\left(\chi_n(a)a+\chi_n\left(\frac{D}{2^k}-a\right)\cdot\left(\frac{D}{2^k}-a\right)
\right)\]
and applying Lemma \ref{sec3-1}, we obtain the following congruence 
\begin{equation}\label{sec3-f1}
  \sum^{\frac{D}{2^k}}_{a=1}\chi_n(a)a\equiv (1-\zeta_{2^k})\sum^{\frac{D}{2^{k+1}}}_{a=1}\chi_n(a)a\mod (1-\zeta_{2^{k-1}}).
\end{equation}
By Lemmas \ref{bernoulli}, \ref{no-D} and (1) of Lemma \ref{lem2-1}, we have
\[L(\chi_n,-1)=\frac{1}{2}\sum^{\frac{D}{2}}_{a=1}\chi_n(a)a\equiv \sum^{\frac{D}{4}}_{a=1}\chi_n(a)a \, ({\rm mod}\, 2).\]
Applying \eqref{sec3-f1} with $k=2$, we obtain
\[L(\chi_n,-1)\equiv(1-\zeta_{2^2})\sum^{\frac{D}{8}}_{a=1}\chi_n(a)a\, ({\rm mod}\, 2).\]
Suppose that for an integer $m$ with $3\leq m\leq n$, we have the congruence
\[L(\chi_n,-1)\equiv \prod^{m-1}_{j=2}(1-\zeta_{2^j})\cdot \sum^{\frac{D}{2^m}}_{a=1}\chi_n(a)a\, ({\rm mod}\,  2).\]
Noting the equality of ideals
\[(1-\zeta_{2^{m-1}})\cdot\prod^{m-1}_{j=2}(1-\zeta_{2^j})\BZ[\zeta_{2^n}]=2\BZ[\zeta_{2^n}], \]
and applying equation \eqref{sec3-f1} with $k=m$, we obtain
\[\prod^{m-1}_{j=2}(1-\zeta_{2^j})\cdot \sum^{\frac{D}{2^m}}_{a=1}\chi_n(a)a\equiv \prod^{m}_{j=2}(1-\zeta_{2^m})\sum^{\frac{D}{2^{m+1}}}_{a=1}\chi_n(a)a\, ({\rm mod}\,  2).\]
Iterating this process, we get
\[L(\chi_n,-1)\equiv \prod^{n}_{k=1}(1-\zeta_{2^k})\cdot\sum^{D/2^{n+1}}_{a=1}\chi_n(a)a
=\prod^{n}_{k=1}(1-\zeta_{2^k})\, ({\rm mod}\, 2).\]
Here we recall that $D=2^{n+2}$, so $\sum^2_{a=1}\chi_n(a)a=1$. Therefore, the proposition follows.

\end{proof}

\subsection{$2$-adic divisibility of Dirichlet $L$-values for $d=1$ and general $m$} 
In this subsection, we determine the $2$-adic valuation of $L(\chi_n,1-m)$ for all $n\geq1$ and all positive even integers $m$. 

We begin by presenting several lemmas that will be used in the proof for both cases: $d=1$ and general $d$. We omit the proof of the following lemma, as it is a standard exercise in calculus.

\begin{lem}\label{sec3-2}
Denote by $\log x$ the natural logarithm of $x$. For every real number $x\geq2$, we have 
(1) $3x-3-\frac{\log x}{\log 2}\geq2$; 
(2) $4x-4-\frac{\log x}{\log 2}\geq3$; (3) if $x\geq3$, $2x-3-\frac{\log x}{\log 2}\geq1$; (4) if $x\geq3$, $3x-4-\frac{\log x}{\log 2}\geq 2$. 
\end{lem}

Recall that $D=2^{n+2}d$, where $d$ is a square free positive odd integer.

\begin{lem}\label{sec3-3}
Let $n,k$ be two positive integers. Assume $k\geq2$.
\begin{enumerate}
  \item [(1)]We have ${\rm ord}_2\left(\frac{D^{k-1}}{k}\right)\geq2$. If $n\geq2$, we have ${\rm ord}_2\left(\frac{D^{k-1}}{k}\right)\geq3$.
  \item [(2)]We have ${\rm ord}_2\left(\frac{D^{k-1}}{k2^k}\right)\geq1$ unless $n=1$ and $k=2$.
  \item [(3)]If $n\geq2$, we have ${\rm ord}_2\left(\frac{D^{k-1}}{k2^k}\right)\geq2$ unless $n=2$ and $k=2$. 
\end{enumerate}
\end{lem}

\begin{proof}
This lemma follows from Lemma \ref{sec3-2}, we only show (3). If $n\geq2,k\geq2$, noting that ${\rm ord}_2(k)\leq \frac{\log k}{\log 2}$, we obtain
\[{\rm ord}_2\left(\frac{D^{k-1}}{k 2^k}\right)\geq 3k-4-\frac{\log k}{\log 2}\geq 2\]
unless $n=k=2$.
\end{proof}

For $\fN=D$ or $\frac{D}{2}$, we define
\[\mss(\fN)=\frac{1}{2}\sum^{\fN}_{a=1}\chi_n(a)a^{m-1}.\]
We also define
\[\mst(D)=\frac{2}{mD}\sum^{\frac{D}{2}}_{a=1}\chi_n(a)a^m.\]

\medskip

\begin{lem}\label{sec3-4}
\noindent
\begin{enumerate}
  \item [(1)]The number $\mss(D)$ is integral and $\mss(D)\equiv 0 ({\rm mod}\, 4)$. 
  \item [(2)] We have $\mst(D)\equiv \mss(\frac{D}{2}) ({\rm mod}\, 2)$. If $n\geq2$, then $\mst(D)\equiv \mss(\frac{D}{2}) ({\rm mod}\, 4)$.
\end{enumerate}
 \end{lem}

\begin{proof}
Note the identity
\[\mss(D)=\frac{1}{2}\sum^{\frac{D}{2}}_{a=1}\chi_n(a)\cdot \left(a^{m-1}+(D-a)^{m-1}\right).\]
Recall that $D=2^{n+2}d$ with $n\geq1$ and $d$ an odd integer. It follows that $\mss(D)$ is $2$-adically integral. Moreover, since $m$ is even,
we obtain 
\[\mss(D)\equiv \frac{1}{2}\sum^{\frac{D}{2}}_{a=1}\chi_n(a)(a^{m-1}+(-a)^{m-1})=0\, ({\rm mod}\, 4).\]
Hence the part (1) follows. For the second part, we prove only the congruence modulo $4$, as the modulo $2$ congruence can be shown similarly. From Lemma \ref{lem2-1}, we have $\chi_n\left(\frac{D}{2}-a\right)=-\chi_n(a)$. Therefore, we get
\[\frac{1}{mD}\sum^{\frac{D}{2}}_{a=1}\chi_n(a)\left(a^m+
\left(\frac{D}{2}-a\right)^m\right)=0,\]
equivalently,
\begin{equation}\label{sec3-f2}
  0=\sum^{\frac{D}{2}}_{a=1}\chi_n(a)\cdot\left(\frac{2}{mD}a^m+
\sum^m_{k=1}\frac{D^{k-1}}{k 2^k}\binom{m-1}{k-1}(-a)^{m-k}\right).
\end{equation}
Consider the terms in the sum in the bracket. From Lemma \ref{sec3-3}, for all $k\geq2$, we have
\[{\rm ord}_2\left(\frac{D^{k-1}}{k 2^k}\right)\geq 2\]
unless $n=k=2$. In this exceptional case, we analyze the term
\[\frac{D}{8}\sum^{\frac{D}{2}}_{a=1}\chi_n(a)\cdot (m-1)a^{m-2}.\]
It turns out to be
\[(2d)\cdot\sum^{\frac{D}{4}}_{a=1}\chi_n(a)\cdot(m-1)\cdot \left(a^{m-2}-\left(2^3d-a\right)^{m-2}\right)\equiv 0\, ({\rm mod}\, 4).\]
From equation \eqref{sec3-f2}, we then deduce
\[\mst(D)\equiv \frac{1}{2}\sum^{\frac{D}{2}}_{a=1}\chi_n(a)a^{m-1}\, ({\rm mod}\, 4).\]
This completes the proof in the case 
$n\geq2$.
\end{proof}

\medskip

Now we come back to the case $d=1$, i.e., $D=2^{n+2}$. 

\begin{prop}\label{sec3-5}
Let $n$ be a positive integer, and let $m$ be a positive even integer. Then $L(\chi_n,1-m)$ is $2$-adically integral, and we have
\[L(\chi_n,1-m)\equiv L(\chi_n,-1)\, ({\rm mod}\, 2).\]
In particular, ${\rm ord}_2(L(\chi_n,1-m))=1-2^{1-n}$.
\end{prop}

\begin{proof}
The key idea is to apply the von Staudt--Clausen theorem to reduce the Bernoulli numbers appearing in the expression for $L(\chi_n,1-m)$ to character sums such as $\mss(D)$ and $\mst(D)$ considered in the above.  

From Lemma \ref{bernoulli}, and formulas \eqref{sec202}, \eqref{sec201}, we get
\[-L(\chi_n,1-m)=\frac{1}{mD}\sum^D_{a=1}\chi_n(a)a^m+\frac{1}{m}
\sum^D_{a=1}\chi_n(a)\sum^m_{k=1}\binom{m}{k}B_k\cdot D^{k-1}\cdot a^{m-k}.\]
This equality can be rewritten as
\begin{equation}\label{sec3001}
-L(\chi_n,1-m)=\frac{1}{mD}\sum^D_{a=1}\chi_n(a)a^m+\sum^D_{a=1}
\chi_n(a)\sum^m_{k=1}\binom{m-1}{k-1}\frac{D^{k-1}}{k}B_k\cdot a^{m-k}.
\end{equation}
By Lemma \ref{sec3-3} (1) and the von Staudt-Clausen Theorem, we know that
\[\sum^m_{k=2}\binom{m-1}{k-1}\frac{D^{k-1}}{k}B_k\cdot a^{m-k}\equiv 0 \, ({\rm mod}\, 2).\] 
Since $B_1=-\frac{1}{2}$, \eqref{sec3001} implies that
\begin{equation}\label{sec3002}
 L(\chi_n,1-m)\equiv -\frac{1}{mD}\sum^D_{a=1}\chi_n(a)a^m+\mss(D)\, ({\rm mod}\, 2).
\end{equation}
We now analyze the term
$\frac{1}{mD}\sum^D_{a=1}\chi_n(a)a^m$.
Note the identity
\[\frac{1}{mD}\sum^D_{a=1}\chi_n(a)a^m=\frac{1}{mD}
\sum^{\frac{D}{2}}_{n=1}\chi_n(a)\left(a^m+(D-a)^m\right).\]
Using the binomial expansion and the same computation method as in \eqref{sec3001}, and applying Lemma \ref{sec3-3} (1), we obtain
\begin{equation}\label{sec3003}
  \frac{1}{mD}\sum^D_{a=1}\chi_n(a)a^m\equiv \mst(D)-2\mss\left(\frac{D}{2}\right)\, ({\rm mod}\, 2).
\end{equation}
Therefore, the formula \eqref{sec3002} becomes to be
\[L(\chi_n,1-m)\equiv -\mst(D)+2\mss\left(\frac{D}{2}\right)+\mss(D)\, ({\rm mod}\, 2).\]
From Lemma \ref{sec3-4}, we obtain
\begin{equation}\label{sec3004}
  L(\chi_n,1-m)\equiv \mss\left(\frac{D}{2}\right)\, ({\rm mod}\, 2).
\end{equation}
By definition,
\[\mss\left(\frac{D}{2}\right)=\frac{1}{2}\sum^{\frac{D}{4}}_{a=1}\chi_n(a)
\left(a^{m-1}-\left(\frac{D}{2}-a\right)^{m-1}\right)\equiv \sum^{\frac{D}{4}}_{a=1}\chi_n(a)a\, ({\rm mod}\, 2).\]
Hence, by \eqref{sec3004}, $L(\chi_n,1-m)$ is $2$-adic integral. Finally, by Corollary \ref{sec3-cor1}, we conclude that 
\[ L(\chi_n,1-m)\equiv L(\chi_n,-1) \, ({\rm mod}\, 2).\]
This finishes the proof of the proposition.
 \end{proof}

\subsection{$2$-adic divisibility of Dirichlet $L$-values for general $d$ and $m$} 
We now turn to the case $d>1$. For an odd prime $p$, we denote by $p^*=\left(\frac{-1}{p}\right)p$ and put $f_p={\rm ord}_2\left(\frac{p^*-1}{4}\right)$.
We will prove the following theorem in this subsection.

\begin{thm}\label{sec3-3-1}
Let $n,d$ be two positive integers. Assume that $d$ is odd square-free and $d\geq2$. Let $\psi_d$ be the Dirichlet character associated to $\BQ(\sqrt{d})/\BQ$. There exists a positive integer $n_d$ such that for all integers $n\geq n_d$, we have 
\[\ord_2(L(\chi_n\psi_d,1-m))=1+2^{1-n}\left(-1+\sum_{p\mid d}2^{f_p}\right).\]
\end{thm}

A key step to prove the above theorem is to show the following lemma, which is a refinement of the proof of Proposition \ref{sec3-5}.

\begin{lem}\label{sec3-3-2}
Let $D=2^{n+2}d$, where $d$ is a positive odd square-free integer. Let $\eta$ denote one of the three characters $\chi_n$, $\chi_n\psi_d$ or $\chi_n\psi^2_d$. Assume that $n\geq2$, then 
$L(\eta,1-m)$ is $2$-adically integral, and 
\[L(\eta,1-m)\equiv \sum^{\frac{D}{4}}_{a=1}\eta(a)a\, ({\rm mod}\, 4).\]
\end{lem}

\begin{proof}
Since the conductor of $\eta$ divides $D$, and using equation \eqref{sec202} and Lemma \ref{bernoulli}, we have
\[L(\eta,1-m)=-\frac{D^{k-1}}{m}\sum^D_{a=1}\eta(a)\sum^m_{k=0}\binom{m}{k}B_k\cdot a^{m-k}.\]
From Lemma \ref{sec3-3} (1), we know that 
${\rm ord}_2\left(\frac{D^{k-1}}{k}\right)\geq3$. By similar methods as in equations \eqref{sec3001}, \eqref{sec3002}, and invoking the von Staudt-Clausen theorem, we obtain
\begin{equation}\label{sec3-3-2-f1}
 L(\eta,1-m)\equiv -\frac{1}{mD}\sum^D_{a=1}\eta(a)a^m-\mss_\eta(D) \, ({\rm mod}\, 4), 
\end{equation}
where for a positive integer $\fN$, we define $\mss_\eta(\fN)=\frac{1}{2}\sum^\fN_{a=1}\eta(a)a^{m-1}$. For the term $\frac{1}{mD}\sum^D_{a=1}\eta(a)a^m$, using the same computation as in  \eqref{sec3003}, we get
\begin{equation}\label{sec3-3-2-f2}
  \frac{1}{mD}\sum^D_{a=1}\eta(a)a^m\equiv \mst_\eta(D)-2\mss_\eta(D/2) \, ({\rm mod}\, 4),
\end{equation}
where we define $\mst_\eta(D)=\frac{2}{mD}\sum^{\frac{D}{2}}_{a=1}\eta(a)a^m$. 
It follows, by the same method as in Lemma \ref{sec3-4}, that
$\mss_\eta(D)\equiv 0\,({\rm mod}\, 4)$ and $\mst_\eta(D)\equiv\mss_\eta(D/2)\, ({\rm mod}\, 4)$. Therefore, from \eqref{sec3-3-2-f1}, we obtain
\[L(\eta,1-m)\equiv \mss_\eta(D/2)\, ({\rm mod}\, 4).\]
Moreover, we observe that
\[\mss_\eta(D/2)=\frac{1}{2}\sum^{\frac{D}{4}}_{a=1}\eta(a)\cdot\left(a^{m-1}-\left(\frac{D}{2}-a\right)^{m-1}\right)\equiv
\sum^{\frac{D}{4}}_{a=1}\eta(a)a^{m-1}\, ({\rm mod}\, 4),\]
thus $L(\eta,1-m)$ is $2$-adically integral. Since $m$ is positive and even, and \(\eta(a)=0\) unless \(a\) is odd, we have $a^{m-1}\equiv a\, ({\rm mod}\, 4)$. Hence, the lemma follows.
\end{proof}

\begin{lem}\label{sec3-3-3}
Let $D=2^{n+2}d$, where $d$ is a positive odd square free integer. Let $\eta$ denote one of the three characters $\chi_n$, $\chi_n\psi_d$ or $\chi_n\psi^2_d$. Assume that $n\geq2$, then
\[L(\eta,1-m)\equiv (1\mp\zeta_4)\sum^{D/8}_{a=1}\eta(a)a\cdot\varepsilon_a\, ({\rm mod}\, 4).\]
where $\varepsilon_a$ is defined in Proposition \ref{mod-4} and $\chi_n(2^nd-1)=\pm\zeta_4$. 
\end{lem}

\begin{proof}
Observe that
\[\sum^{\frac{D}{4}}_{a=1}\eta(a)a=\sum^{\frac{D}{8}}_{a=1}\eta(a)a+\sum^{\frac{D}{8}}_{a=1}\eta\left(\frac{D}{4}-a\right)\cdot\left(\frac{D}{4}-a\right).\]
Note that
\[\chi_n(2^nd-a)=\chi_n(2^n-a)\cdot\left(\frac{-1}{d}\right),\]
and Lemma \ref{lem2-1}, we have $\eta\left(\frac{D}{4}-a\right)=\eta\left(\frac{D}{4}-1\right)\cdot\left(\frac{-1}{a}\right)\cdot\eta(a)$. Applying the same computation as in Proposition \ref{mod-4}, we obtain
\[\sum^{\frac{D}{4}}_{a=1}\eta(a)a\equiv (1\mp\zeta_4)\sum^{\frac{D}{8}}_{a=1}\eta(a)a\cdot \varepsilon_a \, ({\rm mod}\, 4).\]
Then, by Lemma \ref{sec3-3-2}, the result follows. 
\end{proof}

We now prove Theorem \ref{sec3-3-1}.

\begin{proof}[Proof of Theorem \ref{sec3-3-1}.]
We claim that the following congruence holds:
\begin{equation}\label{sec3-3-1-f1}
  L^{(D)}(\chi_n,1-m)+L(\chi_n\psi_d,1-m)\equiv 0\, ({\rm mod}\, 2(1+\zeta_4)).
\end{equation}
Assuming the claim, we derive the theorem.
By definition, we have
\[L^{(D)}(\chi_n,1-m)=L(\chi_n,1-m)\cdot \prod_{p\mid d}\left(1-\chi_n(p)p^{m-1}\right).\]
From Proposition \ref{sec3-5}, we obtain
\[{\rm ord}_2(L^{(D)}(\chi_n,1-m))=1-2^{1-n}+\sum_{p\mid d}{\rm ord}_2\left(\left(1-\chi_n(p)p^{m-1}\right)\right).\]
To compute the $2$-adic valuations of the terms in the sum above, note that
\[1-\chi_n(p)p^{m-1}\equiv 1-\chi_n(p) \, ({\rm mod}\, 2).\]
It is known that, for $n>f_p$, the order of the Frobenius element associated to $p$ in ${\rm Gal}(\BQ_n/\BQ)$ is equal to $2^{n-f_p}$. Therefore, for \(n>f_p+1\) we have $${\rm ord}_2(1-\chi_n(p))=2^{1-n+f_p}<1.$$ 
The congruence above together with the strong triangle inequality for \({\rm ord}_2\) imply that for such \(n\)
\[{\rm ord}_2(1-\chi_n(p)p^{m-1})={\rm ord}_2(1-\chi_n(p))=2^{1-n+f_p}.\]
So, if $n>f_p+1$ for all $p\mid d$, we have
\[{\rm ord}_2 \bigg( \prod_{p\mid d}\left(1-\chi_n(p)p^{m-1}\right) \bigg)=2^{1-n}\cdot \sum_{p\mid d}2^{f_p}.\]
Therefore, 
\[{\rm ord}_2(L^{(D)}(\chi_n,1-m))=1+2^{1-n}\cdot (-1+\sum_{p\mid d}2^{f_p}).\]
Now, invoking the congruence \eqref{sec3-3-1-f1}, when $n$ is sufficiently large, say there exists an integer $n_d$, such that for all $n\geq n_d$,
\[1+2^{1-n}\cdot (-1+\sum_{p\mid d}2^{f_p})<\frac{3}{2},\]
By the claim \eqref{sec3-3-1-f1} and the strong triangle inequality for \({\rm ord}_2\) again, we have
\[{\rm ord}_2(L(\chi_n\psi_d,1-m))={\rm ord}_2(L^{(D)}(\chi_n,1-m))=1+2^{1-n}\cdot (-1+\sum_{p\mid d}2^{f_p}).\]

Now, we prove the claim. Recall from Lemma \ref{imprimitive} that
\[L^{(D)}(\chi_n,1-m)=L(\chi_n\psi^2_d,1-m)
.\]
Apply Lemma \ref{sec3-3-3} to \(\eta=\chi_n\psi_d^2\) and \(\eta=\chi_n\psi_d\) respectively, we get
\[L(\chi_n\psi^2_d,1-m)
\equiv(1\mp\zeta_4)\sum^{D/8}_{a=1}\chi_n(a)\psi_d(a)^2a\varepsilon_a\, ({\rm mod}\, 4),\]
and
\[L(\chi_n\psi_d,1-m)\equiv(1\mp\zeta_4)\sum^{D/8}_{a=1}\chi_n(a)\psi_d(a)a\varepsilon_a\, ({\rm mod}\, 4).\]
Therefore, taking the sum of the congruences above we obtain
\[L^{(D)}(\chi_n,1-m)+L(\chi_n\psi_d,1-m)\equiv (1\mp\zeta_4)\sum^{D/8}_{a=1}\chi_n(a)\psi_d(a)\cdot a\varepsilon_a\cdot(1+\psi_d(a))\, ({\rm mod}\, 4).\]
Note that $2$ divides $1+\psi_d(a)$. Thus, 
the right hand side is divisible by $2(1+\zeta_4)$, and the claim follows.
\end{proof}

For a real number $x$, we define $\lceil x \rceil={\rm min}\{n\in \mathbb{Z}:n\ge x\}$. Concerning the number $n_d$, we have the following result.

\begin{cor}\label{nd}
Assume that $d>1$. Let $f=\max\limits_{p|d} \{f_p\}$. Then we can take
\[n_d=\lceil f+\log_2(\tau(d))+2 \rceil, \]
where $\tau(d)$ is the number of distinct prime divisors of $d$. 
\end{cor}

\begin{proof}
From the proof of Theorem \ref{sec3-3-1}, we require that for $n\ge n_d$, the inequality 
\[1+2^{1-n}\cdot\left(-1+\sum\limits_{p|d}2^{f_p}\right)<\frac{3}{2}\]
holds.
Since 
\[\sum\limits_{p|d}2^{f_p} \le \tau(d)\cdot 2^{f},\]
therefore, we can take
\[n_d=\lceil f+\log_2(\tau(d))+2 \rceil. \]
\end{proof}

\subsection{$2$-adic divisibility of Dedekind zeta-values}

Let $\zeta(s)$ denote the Riemann zeta function. 

\begin{lem}\label{zeta-1}
For each positive even integer $m$, we have
\[{\rm ord}_2(\zeta(1-m))=-1-{\rm ord}_2(m).\]
\end{lem}

\begin{proof}
Note that $\zeta(1-m)=-\frac{B_m}{m}$. By the von Staudt-Clausen theorem, we obtain ${\rm ord}_2(B_m)=-1$. Therefore,  the result follows.
\end{proof}

Let $d>1$ be a positive odd square-free integer. Let $K=\BQ(\sqrt{d})$. For each non-negative integer $n$, let $K_n$ denote the unique subfield in $K_{cyc}$ such that $[K_n:K]=2^n$. Note that $K_n=\BQ_n(\sqrt{d})$, 
where $\BQ_n$ denotes the $n$-th layer of the cyclotomic $\BZ_2$-extension of $\BQ$. 

\begin{thm}\label{zeta-2}
Let $m$ be a positive even integer and let $n$ be a non-negative integer.
\begin{enumerate}
  \item [(1)]We have
  \[{\rm ord}_2(\zeta_{\BQ_n}(1-m))=2^n-n-2-{\rm ord}_2(m).\]
  \item [(2)]For all $n\geq n_d$, we have
  \[{\rm ord}_2(\zeta_{K_n}(1-m))=2\cdot 2^n+\left(-2+\sum_{p\mid d}2^{f_p}\right)\cdot n+\nu'_{m,d},\]
  where
 \[\begin{aligned}\nu'_{m,d}=&\ord_2\left(\frac{L(\psi_d,1-m)}{4m}\right)-2^{n_d}+n_d\left(1-\sum_{p\mid d}2^{f_p}\right)\\
 &+\sum^{n_d}_{k=1}2^{k-1}\cdot{\rm ord}_2(L(\chi_k\psi_d,1-m)).\end{aligned}\] 
\end{enumerate}

\end{thm}

\begin{proof}
For (1), from the formula
\[\zeta_{\BQ_n}(1-m)=\prod^n_{\ell=0}\prod_{\chi_\ell}L(\chi_\ell,1-m),\]
where the second product is taken over 
all primitive characters $\chi_\ell$ of ${\rm Gal}(\BQ_\ell/\BQ)$.
From Lemma \ref{zeta-1}, ${\rm ord}_2(\zeta(1-m))=-1-{\rm ord}_2(m)$. We apply Proposition \ref{sec3-5} to obtain
\[{\rm ord}_2(\zeta_{\BQ_n}(1-m))=2^n-n-2-{\rm ord}_2(m).\]
This proves (1).

Since $K_n=\BQ_n(\sqrt{d})$, we have the formula
\[\zeta_{K_n}(1-m)=\zeta_{\BQ_n}(1-m)\cdot\prod^n_{\ell=0}\prod_{\chi_\ell}L(\chi_\ell\psi_d,1-m),\]
where the second product is over all primitive characters $\chi_\ell$ of ${\rm Gal}(\BQ_\ell/\BQ)$. From Theorem \ref{sec3-3-1}, for $n\geq n_d$, we obtain
\[{\rm ord}_2(L(\chi_n\psi_d,1-m))=1+2^{1-n}\left(-1+\sum_{p\mid d}2^{f_p}\right).\]
Combining this with the formula from part (1), and evaluating 
the remaining terms directly, we conclude that the formula in part (2) follows.  
\end{proof}

\section{Structure of the $K$-groups}\label{sec4}

For a number field $\CF$ and a positive even integer $m$, we define $w_m(\CF)$ to be the largest integer $\fN$ such that ${\rm Gal}(\CF(\zeta_\fN)/\CF)$ has exponent dividing $m$. We will use the following theorem to determine the structure of the $K$-groups in this section. 

\begin{thm}[\cite{RW}]\label{4-1}
Let $F$ be a totally real abelian number field, for any positive even integer $m$, we have
\[{\rm ord}_2\left(|K_{2m-2}\CO_F|\right)=\begin{cases}{\rm ord}_2(w_m(F)\zeta_F(1-m))\quad &\textrm{if ${\rm ord}_2(m)=1$};\\ {\rm ord}_2(w_m(F)\zeta_F(1-m))-[F:\BQ]\quad &\textrm{otherwise.}
\end{cases}\]
\end{thm}

Recall that $K=\BQ(\sqrt{d})$, where $d\geq2$ is a positive square free odd integer.

\begin{lem}\label{4-2}
Let $F=\BQ_n$ or $K_n$. For all nonnegative integer $n$, we have
\[{\rm ord}_2(w_m(K_n))={\rm ord}_2(w_m(\BQ_n))=n+2+{\rm ord}_2(m).\]
\end{lem}

\begin{proof}
By assumption, for each positive integer $j$, it is shown that 
\[{\rm Gal}(F(\zeta_{2^{n+2+j}})/F)\simeq \BZ/2\BZ\oplus \BZ/2^j\BZ.\]
By the definition of $w_m(F)$, we obtain $${\rm ord}_2(w_m(F))=n+2+{\rm ord}_2(m).$$
Hence, the lemma follows.
\end{proof}

Recall the number $n_d$ defined in Corollary \ref{nd}. We set 
\[n_K:=\lceil f+\log_2(\tau(d))+2 \rceil.\] 
Now we can prove the main theorem.

\begin{proof}[Proof of Theorem \ref{main-thm1}]
Let $F=\BQ$ or $K$. From Theorems \ref{zeta-2}, \ref{4-1}, and Lemma \ref{4-2}, for all $n\geq n_K$ (resp. $n\geq1$) when $F=K$ (resp. $F=\BQ$), we have
\[{\rm ord}_2\left(|K_{2m-2}\CO_{F_n}(2)|\right)=\mu\cdot 2^n+\lambda\cdot n+\nu,\]
where 
\begin{itemize}
  \item $\mu=0$ if $4\mid m$; $\mu=[F:\BQ]$ if ${\rm ord}_2(m)=1$,
  \item $\lambda=1-[F:\BQ]+\sum_{p\mid d}2^{f_p}$,
  \item $\nu=\nu'_{m,d}+2+{\rm ord}_2(m)$.
\end{itemize}
Here, $\nu'_{m,d}$ is the number defined in Theorem \ref{zeta-2}. The theorem now follows. 
\end{proof}

From now until the end of this section, we assume $m=2$. We will prove a refinement of $n_d$ in section \S\ref{sec5}.

\begin{prop}\label{4-3}
Assume that $m=2$. For $d\geq2$, we can take
\[n_d={\rm max}_{p\mid d}\{f_p\}+2.\]
Namely, for any $n\geq n_d$, we have
\[{\rm ord}_2(L(\chi_n\psi_d,-1))=1+2^{1-n}\cdot\left(-1+\sum_{p\mid d}2^{f_p}\right)\]
\end{prop}

For a finite abelian group $\fA$, we denote by $r_2(\fA)$ the $2$-rank of $\fA$, i.e., the $\BF_2$ dimension of the vector space $\fA/2\fA$.
The following two examples complete the proof of the Corollary \ref{cor1-1}.

\begin{exm}\label{4-4}
By the Theorem \ref{main-thm1}, we have
\[{\rm ord}_2\left(|K_2(\CO_{\BQ_n})|\right)=2^n\qquad\textrm{for all $n\geq0$.}\]
We recall a theorem of Browkin \cite[Theorem 3]{Browkin}. Let $\CF$ be a number field, $r_1$ be the number of real embeddings of $\CF$, $g_2$ the number of distinct prime ideals in $\mathcal{O}_\CF$ above $2$, and $r=r_2(Cl(F)/Cl_2(F))$, where $Cl_2(F)$ is the subgroup generated by the ideal classes of prime ideals in $\mathcal{O}_\CF$ above $2$. Then 
\begin{equation}\label{4-4-f1}
  r_2(K_2\mathcal{O}_\CF)=r_1+g_2-1+r.
\end{equation}
In the case $\CF=\mathbb{Q}_n$, we have $r_1=2^n$, $g_2=1$. Since
\[2^n=r_1+g_2-1\le r_2(K_2\mathcal{O}_{\mathbb{Q}_n})\le {\rm ord}_2(|K_2\mathcal{O}_{\mathbb{Q}_n}(2)|)=2^n,\]
we obtain
\[K_2\mathcal{O}_{\mathbb{Q}_n}(2)\simeq (\mathbb{Z}/2\mathbb{Z})^{2^n}.\]  
\end{exm}

\begin{exm}\label{4-5}
Let $p\equiv \pm3\, ({\rm mod}\, 8)$ be a prime. Note that \cite[Example 6 and case $n=1$ in Proposition 9]{DL}, we have 
\[{\rm ord}_2(L(\psi_{p},-1))=1\]
From \cite[Corollary]{BS}, for $\fF=\BQ(\sqrt{2p})$, we have $K_2\CO_{\fF}(2)\simeq (\BZ/2\BZ)^2$. Since 
$${\rm ord}_2(w_2(\fF))=3$$ 
(see \cite{DL}), the Birch-Tate formula implies
\[{\rm ord}_2(L(\chi_1\psi_p,-1))=1.\]
From Theorem \ref{sec3-3-1} and Proposition \ref{4-3}, we then have
\[{\rm ord}_2(L(\chi_n\psi_p,-1))=1\quad\textrm{for all $n\geq2$.}\]
 Therefore, by the Birch-Tate formula, for the real quadratic number fields $K=\BQ(\sqrt{p})$ or $\BQ(\sqrt{2p})$, we obtain
\[{\rm ord}_2(|K_2\CO_{K_n}|)=2\cdot 2^n.\]
Recalling formula \eqref{4-4-f1} and applying an inequality similar to that in Example \ref{4-4}, we find $r_2(K_2\CO_{K_n})=2^{n+1}$ and
\[K_2\CO_{K_n}(2)\simeq (\BZ/2\BZ)^{2^{n+1}}.\] 
\end{exm}

We conclude this section by proving Example \ref{main-thm2}.

\begin{proof}[Proof of Example \ref{main-thm2}]
From \cite[Proposition 9]{DL}, we have
\[{\rm ord}_2(L(\psi_d,-1))=r,\]
where $d=p_1p_2\cdots p_r$. From Theorem \ref{sec3-3-1} and Proposition \ref{4-3}, we have
\[{\rm ord}_2(L(\chi_n\psi_d,-1))=1+2^{1-n}\cdot(r-1)\quad \textrm{for all $n\geq2$.}\]
Therefore, $\nu'_{2,d}={\rm ord}_2(L(\psi_{2d},-1))-4$. Now applying Theorem \ref{main-thm1}, we obtain 
\[\mu=2,\quad\lambda=r-1,\quad\nu={\rm ord}_2(L(\psi_{2d},-1))-1.\]
The result follows.
\end{proof}

\section{A Refined bound for $n_d$}\label{sec5}

We prove Proposition \ref{4-3} in this section.
From Corollary \ref{nd}, we have seen that for $d>1$, we can take
\[n_d=\lceil f+\log_2(\tau(d))+2 \rceil, \]
where $\tau(d)$ denotes the number of distinct prime divisors of $d$ and
$f=\max\limits_{p|d} \{f_p\}$. 

\begin{prop}\label{prop5-1}
Assume that $m=2$.
For $d>1$, we can take
$n_d=f+2$. Namely, for any $n\geq f+2$, we have
\begin{equation}\label{prop5-1-f1}
  {\rm ord}_2(L(\chi_n\psi_d,-1))=1+2^{1-n}\cdot\left(-1+\sum_{p\mid d}2^{f_p}\right).
\end{equation}
\end{prop}

Our strategy is to prove that equation \eqref{prop5-1-f1} holds for all $n\geq f+2$.
The key idea is an induction method in \cite{DL}.

\begin{proof}
We proceed by induction on 
$\tau=\tau(d)$. When $\tau=1$, the result holds from previous arguments (see Corollary \ref{nd}), and we can take $n_d=f+2$ as before. Suppose the statement holds for all integers $d'$ with 
$\tau(d')<\tau$, and now we consider $d$ with $\tau(d)=\tau\geq 2$. 

By Lemma \ref{no-D}, for $n\ge 2$, we have
\[\begin{aligned}
\sum\limits_{b|d} L^{(D)}(\chi_n\psi_b,-1)= & \sum\limits_{b|d} \frac{1}{2} \sum\limits_{a=1}^{D/2} \chi_n(a)\psi_d(a)\psi_{d/b}(a)a \\
=& \frac{1}{2} \sum\limits_{a=1}^{D/2} \chi_n(a)\psi_d(a)a \prod\limits_{p|d} (1+\psi_p(a))\\
=&\sum\limits_{a=1}^{D/4} \chi_n(a)\psi_d(a)\left(a-\frac{1}{4}D\right) \prod\limits_{p|d} (1+\psi_p(a)).
\end{aligned}\]
Since ${\rm ord}_2\left( \prod\limits_{p|d} (1+\psi_p(a)) \right)\geq \tau$ for $a$ coprime to $d$, we obtain
\begin{equation}\label{5-1-f1}
  {\rm ord}_2\left(\sum\limits_{b|d} L^{(D)}(\chi_n\psi_b,-1)\right)\geq \tau.
\end{equation}
Now observe the formula:
\[{\rm ord}_2\left(L^{(D)}(\chi_n\psi_b,-1)\right)=\sum\limits_{p|(d/b)} {\rm ord}_2(1-\chi_n(p)\psi_b(p)p) +{\rm ord}_2(L(\chi_n\psi_b,-1)).\]
As shown in the proof of Theorem \ref{sec3-3-1}, for $n\geq f_p+2$, we have 
$${\rm ord}_2(1-\chi_n(p)\psi_b(p)p)=2^{1-n+f_p}.$$
 By induction hypothesis, for all $b<d$ and $n\geq f+2$, we have
\[{\rm ord}_2\left(L^{(D)}(\chi_n\psi_b,-1)\right)=1+2^{1-n}\left(-1+\sum_{p\mid d}2^{f_p}\right).\]
Since $n\geq f+2$ and
\[\sum_{p\mid d}2^{f_p}\leq 2^f\cdot\tau,\]
it can be shown that
\[1+2^{1-n}\left(-1+\sum_{p\mid d}2^{f_p}\right)<\tau.\]
Now, from the inequality \eqref{5-1-f1}, since all terms in the sum except the term $L(\chi_n\psi_d,-1)$, have the same $2$-adic valuation which is strictly less than $\tau$, we conclude that
\[{\rm ord}_2\left(L(\chi_n\psi_d,-1)\right)=1+2^{1-n}\left(-1+\sum\limits_{p|d}2^{f_p}\right)\]
for all $n\geq f+2$. This completes the proof. 
\end{proof}

\end{document}